\documentclass[12pt,reqno]{amsart}
\usepackage{amsmath, amsfonts, amssymb, amsthm, amscd, amsbsy}
\usepackage{fancyhdr}
\usepackage[usenames,dvipsnames,svgnames,x11names,hyperref]{xcolor}
\usepackage{geometry}
\usepackage{graphicx}
\usepackage[pagebackref]{hyperref}%
\usepackage{amsmath}%
\usepackage{amsfonts}%
\usepackage{amssymb}
\usepackage{enumerate}

\providecommand{\U}[1]{\protect\rule{.1in}{.1in}}
\hypersetup{backref=true,
pagebackref=true,
hyperindex=true,
colorlinks=true,
breaklinks=true,
urlcolor=NavyBlue,
linkcolor=Fuchsia,
bookmarks=true,
bookmarksopen=false,
filecolor=black,
citecolor=ForestGreen,
linkbordercolor=red
}

\allowdisplaybreaks

\newtheorem{theorem}{Theorem}[section]
\theoremstyle{plain}

\newtheorem{corollary}{Corollary}[section]

\newtheorem{lemma}{Lemma}[section]

\newtheorem{remark}{Remark}

\numberwithin{equation}{section}

\def\vint{\mathop{\mathchoice%
          {\setbox0\hbox{$\displaystyle\intop$}\kern 0.22\wd0%
           \vcenter{\hrule width 0.6\wd0}\kern -0.82\wd0}%
          {\setbox0\hbox{$\textstyle\intop$}\kern 0.2\wd0%
           \vcenter{\hrule width 0.6\wd0}\kern -0.8\wd0}%
          {\setbox0\hbox{$\scriptstyle\intop$}\kern 0.2\wd0%
           \vcenter{\hrule width 0.6\wd0}\kern -0.8\wd0}%
          {\setbox0\hbox{$\scriptscriptstyle\intop$}\kern 0.2\wd0%
           \vcenter{\hrule width 0.6\wd0}\kern -0.8\wd0}}%
          \mathopen{}\int}

\begin{document}
\title[Stability of HLS inequality with lower bounds]{Stability of Hardy-Littlewood-Sobolev inequalities with  explicit lower bounds}
\author{Lu Chen}
\address[Lu Chen]{ School of Mathematics and Statistics, Beijing Institute of Technology, Beijing
100081, PR China}
\email{chenlu5818804@163.com}

\author{Guozhen Lu}
\address[Guozhen Lu]{Department of Mathematics, University of Connecticut, Storrs, CT 06269, USA}
\email{guozhen.lu@uconn.edu}

\author{Hanli Tang}
\address[Hanli Tang]{Laboratory of Mathematics and Complex Systems (Ministry of Education), School of Mathematical Sciences, Beijing Normal University, Beijing, 100875, China}
\email{hltang@bnu.edu.cn}
\thanks{}
%\date{ }

%\subjclass[2010]{42B35, 42B37}

\begin{abstract}
In this paper, we establish the stability for the Hardy-Littlewood-Sobolev (HLS) inequalities with explicit lower bounds. By establishing the relation between the stability of HLS inequalities and the stability of fractional Sobolev inequalities, we also give the stability of the higher and fractional order Sobolev inequalities with the lower bounds. This extends to some extent the stability of the first order Sobolev inequalities with the explicit lower bounds established by Dolbeault, Esteban, Figalli, Frank and Loss in \cite{DEFFL} to the higher and  fractional order case. Our proofs are based on the competing symmetries, the continuous Steiner symmetrization inequality for the HLS integral and the dual stability theory. As another application of the stability of the HLS inequality, we also establish the stability of Beckner's \cite{Beckner1} restrictive Sobolev inequalities of fractional order $s$ with $0<s<\frac{n}{2}$ on the flat sub-manifold $\mathbb{R}^{n-1}$  and the sphere $\mathbb{S}^{n-1}$ with the explicit lower bound.
When $s=1$, this implies the  explicit lower bound for the stability of  Escobar's  first order Sobolev trace inequality  \cite{Escobar} which has remained unknown in the literature.
\end{abstract}

\maketitle
\section{Introduction}
The main purpose of this paper is to give the lower bound estimate of the sharp constants of the stability of the Hardy-Littlewood-Sobolev (HLS) inequality. As  applications, we establish the stability of the  higher and fractional Sobolev inequalities and Sobolev trace type inequalities with explicit lower bounds.
\medskip

The classical Hardy-Littlewood-Sobolev (HLS) inequality (\cite{HL,S}) states
\begin{equation}\label{HL1}
\int_{\mathbb{R}^n}\int_{\mathbb{R}^n}|x-y|^{-(n-2s)} f(x)g(y)dxdy\leq C_{n,p,q'}\|f\|_{L^{q'}(\mathbb{R}^n)}\|g\|_{L^p(\mathbb{R}^n)},
\end{equation}
with $1<q'$, $p<\infty$, $0<s<\frac{n}{2}$ and $\frac{1}{q'}+\frac{1}{p}-\frac{2s}{n}=1$. Lieb and Loss \cite{LiebLoss} applied the layer cake representation formula to give an explicit upper bound for the sharp constant $C_{n,p,q'}$. More precisely, they showed that the best constant
$C_{n,p,q'}$ satisfies the following estimate
$$C_{n,p,q'}\leq \frac{n}{n-\lambda}\Big(\frac{\pi^{\frac{\lambda}{2}}}{\Gamma(1+\frac{n}{2})}\Big)^{\frac{\lambda}{n}}\frac{1}{q'p}
\Big((\frac{\lambda q'}{n(q'-1)})^{\frac{\lambda}{n}}+(\frac{\lambda p}{n(p-1)})^{\frac{\lambda}{n}}\Big).$$
In the special diagonal case $q'=p=\frac{2n}{n+2}$ ($s=1$), Aubin \cite{Au} and Talenti \cite{Ta} derived the sharp constants of HLS inequality by classifying the extremal of classical Sobolev inequality which is a dual form of HLS inequality.
For the HLS inequality of general diagonal case, Lieb \cite{Lieb} classified the extremal function of HLS inequality and obtained the best constant
$$C_{n,p,q'}=\pi^{\frac{\lambda}{2}}\frac{\Gamma(\frac{n}{2}-
\frac{\lambda}{2})}{\Gamma(n-\frac{\lambda}{2})}\Big(\frac{\Gamma(n)}{\Gamma(\frac{n}{2})}\Big)^{1-\frac{\lambda}{n}}.$$
Recently, the authors of \cite{CCL, FL0, FL1, FL2} developed a rearrangement-free argument to obtain the sharp HLS inequality.
Weighted HLS inequalities, also known as the Stein-Weiss inequalities \cite{SteinWeiss},  and their sharp constants and existence of extremal functions have also been studied in Euclidean spaces and the Heisenberg group, see e.g., \cite{Be, CLT1, Han0, HLZ}.

\medskip

The stability of the fractional Sobolev inequality states that there exists some positive constant $C(n,s)$ such that for any $U\in \dot{H}^s(\mathbb{R}^n)$, there holds
\begin{equation*}
	\label{eq:sob}
	 S_{n,s}\left\|(-\Delta)^{s/2} U \right\|_2^2 -\| U\|_{\frac{2n}{n-2s}}^2 \geq C(n,s)\inf\limits_{h\in M_{s}}\|(-\Delta)^{s/2} (U-h)\|^2_{2},
\end{equation*}
where $\dot H^s(\mathbb{R}^n)$ denotes the $s$-order homogenous Sobolev space in $\mathbb{R}^n$: the completion of $C_{c}^{\infty}(\mathbb{R}^n)$ under the norm$\big(\int_{\mathbb{R}^n}|(-\Delta)^{\frac{s}{2}}U|^2dx\big)^{\frac{1}{2}}$,
$$M_s:=\left\{c(\frac{2b}{b^2+|x-a|^2})^{\frac{n-2s}{2}}, a\in \mathbb{R}^n, b>0, c\in \mathbb{R}\right\}$$
is the extremal set of the fractional Sobolev inequality and \begin{equation}
	\label{eq:sobconst}
	S_{n,s} = \Big((4\pi)^s \ \frac{\Gamma(\frac{n+2s}{2})}{\Gamma(\frac{n-2s}{2})} \left( \frac{\Gamma(\frac n2)}{\Gamma(n)} \right)^{2s/n}\Big)^{-1}
	= \Big(\frac{\Gamma(\frac{n+2s}{2})}{\Gamma(\frac{n-2s}{2})} \ |\mathbb{S}^n|^{2s/n}\Big)^{-1} \
\end{equation}
is the sharp constant of the fractional Sobolev inequality. This kind of stability inequality was  proposed in Brezis and Lieb's work in \cite{BrLi}. Bianchi and Egnell in \cite{BiEg} first obtained the stability of the first order Sobolev inequality (i.e., $s=1$). Stabilities of Sobolev inequalities for $s=2$, even positive integers $s<\frac{n}{2}$ and all fractional $0<s<\frac{n}{2}$ were established by the second author and Wei \cite{LuWe}, by Bartsch, Weth and Willem \cite{BaWeWi}, and by Chen, Frank and Weth \cite{ChFrWe} respectively.   It should be noted that although the fractional Sobolev inequality is equivalent to the HLS inequality of diagonal case, their stabilities are not fully equivalent. Carlen \cite{Car} developed a duality theory of the stability  based on the quantitative convexity and deduced the following stability of HLS inequality from the stability of fractional Sobolev inequality:
There exists some constant $C_0(n,s)$ such that for any $g\in L^{\frac{2n}{n+2s}}(\mathbb{R}^n)$, there holds
$$\|g\|^2_{\frac{2n}{n+2s}}-S_{n,s}^{-1}\|(-\Delta)^{-s/2}g\|_2^2\geq C_0(n,s)\inf\limits_{h\in M_{HLS}}\|g-h\|^2_{\frac{2n}{n+2s}},$$
where $$ M_{HLS}:=\left\{c(\frac{2b}{b^2+|x-a|^2})^{\frac{n+2s}{2}}, a\in \mathbb{R}^n, b>0, c\in \mathbb{R}\right\}.$$
\medskip

Stability of Sobolev inequality is proved by establishing the local stability of Sobolev inequality based on the spectrum analysis of elliptic or high-order elliptic operator and using the Lions concentration compactness technique to obtain global stability of Sobolev inequality. (see \cite{BiEg},   \cite{LuWe},  \cite{BaWeWi} and  \cite{ChFrWe}.) However, this method does not tell us any explicit lower bound information about the sharp constant of stability of Sobolev inequality. Recently, Dolbeault, Esteban, Figalli, Frank and Loss in \cite{DEFFL} for the first time obtained the stability for first order Sobolev inequality with the explicit lower bound. More precisely, they proved that for any $f\in \dot{H}^1(\mathbb{R}^n)$, there holds
\begin{equation*}
	\label{eq:sob}
	 S_{n,1}\left\|\nabla f \right\|_2^2 -\| f\|_{\frac{2n}{n-2}}^2 \geq C(n,1)\inf\limits_{h\in M_{1}}\|\nabla (f-h)\|^2_{2},
\end{equation*}
with $C(n,1)\gtrsim \frac{1}{n}$ as $n\rightarrow +\infty$. As an application, they established the stability for Gaussian log-Sobolev inequality with explicit constant. In their proof, there are two  key points. One is to establish the relation between the local stability and global stability of the  Sobolev inequality for positive functions. This needs to use the
the competing symmetries, the continuous Steiner symmetrization inequality (Poly\'{a}-Szeg\"{o} inequality) for the $L^2$ integral of the gradient of $u$.
The other is to establish the relation between the global stability of Sobolev inequality for positive functions and general functions. This is involved with the elementary inequality $\|\nabla u\|_{L^2}^2=\|\nabla u^{+}\|_{L^2}^2+\|\nabla u^{-}\|_{L^2}^2$ for the first order gradient. Naturally, if we would like to establish the stability of fractional Sobolev inequality with explicit lower bounds, we must first establish the relation between the local stability of fractional Sobolev inequality and the global stability of fractional Sobolev inequality for positive functions. Unfortunately, the continuous Steiner symmetrization inequality (Poly\'{a}-Szeg\"{o} inequality) for the $L^2$ integral of the gradient $u$ is not applicable to the $L^2$ integral of $(-\Delta)^{\frac{s}{2}}u$ for $s>1$. Furthermore, the invalidity of $\|(-\Delta)^{s/2} u\|_{L^2}^2=\|(-\Delta)^{s/2}(u^{+})\|_{L^2}^2+\|(-\Delta)^{s/2}(u^{-})\|_{L^2}^2$ for any $s>1$ also adds much challenge of establishing the relation between the global stability of fractional Sobolev inequality for positive functions and general functions. In order to overcome these difficulties, instead of directly considering the stability of fractional Sobolev inequality, we will first establish the stability for the Hardy-Littlewood-Sobolev (HLS) inequality (dual of the fractional Sobolev inequality).  This is possible to be achieved because the continuous Steiner symmetrization inequality (Poly\'{a}-Szeg\"{o} inequality) for the HLS integral $$\int_{\mathbb{R}^n}\int_{\mathbb{R}^n}f(x)|x-y|^{-{n-2s}}f(y)dxdy$$ and the elementary inequality
$\int_{\mathbb{R}^n}|u|^{\frac{2n}{n+2s}}dx=\int_{\mathbb{R}^n}|u^{+}|^{\frac{2n}{n+2s}}dx+\int_{\mathbb{R}^n}|u^{-}|^{\frac{2n}{n+2s}}dx$ hold.
Finally, we will apply the technique of dual stability of functional inequality established by Carlen in \cite{Car} to derive the stability of
of fractional Sobolev inequality with the explicit lower bound.
\vskip0.1cm

Denote by $$K_{n,s}=\sup\limits_{0<\delta<1}\frac{\delta}{2}\frac{n-2s}{n+2s}\min\{m(2\delta)\frac{n-2s}{n+2s},1\},\ \l(\delta)=\sqrt{\frac{\delta}{1-\delta}}, $$ where
$$m(\delta)=\frac{4s}{n+2s+2}-
\frac{2}{2^\ast_s}\sum_{k=3}^{2^\ast_s}\frac{2^{\ast}_s(2^\ast_s-1)\cdots(2^\ast_s-k+1)}{k!}l(\delta)^{k-2},$$
if $2^\ast_s=\frac{2n}{n-2s}$ is an integer
and
$$m(\delta)=\frac{4s}{n+2s+2}-
\frac{2}{2^\ast_s}\sum_{k=3}^{[2^\ast_s]}\frac{2^{\ast}_s(2^\ast_s-1)\cdots(2^\ast_s-k+1)}{k!}l(\delta)^{k-2}-
\frac{2}{2^\ast_s}l(\delta)^{2^\ast_s-2},$$
if  $2^\ast_s=\frac{2n}{n-2s}$ is not an integer.

\medskip
 We first establish the stability of HLS inequality with explicit lower bounds.

\begin{theorem}\label{thm1}
For $0<s<\frac{n}{2}$, let $\mathcal{S}_{HLS}(g)$ denote the HLS stability functional  given by
\begin{align}\label{sta hls func}
\mathcal{S}_{HLS}(g)=\frac{\|g\|^2_{\frac{2n}{n+2s}}-S_{n,s}^{-1}\|(-\Delta)^{-s/2}g\|_2^2}{\inf\limits_{h\in M_{HLS}}\|g-h\|^2_{\frac{2n}{n+2s}}}.
\end{align}
Then there holds
$$\inf\limits_{g\in L^{\frac{2n}{n+2s}}(\mathbb{R}^n)\setminus M_{HLS}} \mathcal{S}_{HLS}(g)\geq \frac{1}{2}\min\{K_{n,s},\min\{2^{\frac{n+2s}{n}}-2,1\}\}.$$
\end{theorem}
\begin{remark}
We note that Carlen \cite{Car} obtained the stability of HLS inequality. However, the explicit lower bound of the stability for HLS inequality was unknown in \cite{Car}. In Theorem \ref{thm1}, we give the stability of HLS inequality with the explicit lower bound.
\end{remark}
\begin{remark}
We also note that the stability of HLS with the explicit lower bound is  important and very useful because  the explicit lower bounds for the stabilities of some important geometric and functional inequalities such as the fractional Sobolev inequalities and Sobolev trace inequalities.
\end{remark}

The proof of Theorem \ref{thm1} is rather involved. For reader's convenience, we will give below an outline of its proof. The proof is divided into four steps.
\vskip0.1cm

{\bf Step 1.} We will first establish the quantitative estimate of the local stability of fractional Sobolev inequality for non-negative functions. Define
\begin{equation*}\begin{split}
\nu(\delta)=\inf\left\{\mathcal{S}_{S}(f): f\geq 0, \inf_{g\in M_s}\|(-\Delta)^{s/2}(f-g)\|_2^2\leq \delta\|(-\Delta)^{s/2}f\|_2^2\right\},
\end{split}\end{equation*}
where $$\mathcal{S}_{S}(f)=\frac{S_{s,n}\|(-\Delta)^{s/2}f\|_2^2-\|f\|^2_{2^\ast_s}}{\inf\limits_{h\in M_{s}}\|(-\Delta)^{s/2}(f-h)\|^2_{2}}.$$
We apply the refined estimate in the neighborhood of the Aubin-Talenti functions and spectrum analysis of the related operators to prove that
$$\nu(\delta)\geq S_{n,s}\left\{\frac{4s}{n+2s+2}-
\frac{2}{2^\ast_s}\sum_{k=3}^{[2^\ast_s]}\frac{2^{\ast}_s(2^\ast_s-1)\cdots(2^\ast_s-k+1)}{k!}l(\delta)^{k-2}-
\frac{2}{2^\ast_s}l(\delta)^{2^\ast_s-2}\right\}.$$

   {\bf Step 2.} We then establish the relation between the local stability for HLS inequality and fractional Sobolev inequality for non-negative functions. This will be done by applying the duality method of Carlen \cite{Car} to our situation. Define
$$\mu(\delta)=\inf\{\mathcal{S}_{HLS}(g): g>0, \inf_{h\in M_{HLS}}\|g-h\|^2_{\frac{2n}{n+2s}}\leq \delta\|g\|^2_{\frac{2n}{n+2s}}\},$$
we will prove that $$\mu(\frac{\delta}{2})\geq \frac{1}{2}\frac{n-2s}{n+2s}\min\{\frac{\nu(\delta)}{S_{n,s}}\frac{n-2s}{n+2s},1\}.$$

{\bf Step 3.} We establish the relationship between the local stability and the global stability of HLS inequality for positive functions.
  On the one hand, our proof is inspired by the method in  Dolbeault, Esteban, Figalli, Frank and Loss's work about the relationship between the local stability and the global stability of fractional Sobolev inequality for positive functions in \cite{DEFFL}. On the other hand,  we must apply the rearrangement flow technique for HLS integral to replace the rearrangement flow technique for $L^2$ integral of gradient because such a rearrangement flow technique is not applicable to $L^2$ integral of higher and fractional order gradient.   More precisely, we will prove that for any $0\leq g \in L^{\frac{2n}{n+2s}}(\mathbb{R}^n)$ satisfying $\inf\limits_{h\in M_{HLS}}\|g-h\|^2_{\frac{2n}{n+2s}}> \delta\|g\|^2_{\frac{2n}{n+2s}}$, there holds
$$\mathcal{S}_{HLS}(g)\geq \delta \mu(\delta).$$
Combining the results in Step 1 to Step 3, we derive the stability of HLS inequality with the explicit lower bound for positive functions.
\vskip 0.1cm

{\bf Step 4.} We finally establish the stability of HLS inequality with the explicit lower bound for general functions. This is achieved by establishing the relation between the stability of HLS inequality for positive functions and general functions. Denoting by $C_{HLS}$ and $C^{pos}_{HLS}$ the optimal constants of stability of HLS inequality for positive functions and general functions respectively, we prove that
$$C_{HLS}\geq \frac{1}{2}\min\{C^{pos}_{HLS},\min\{2^{\frac{n+2s}{n}}-2,1\}\}.$$
\medskip

By the dual stability theory and the stability of HLS inequality with the explicit lower bounds (Theorem \ref{thm1}), we will establish the following stability of fractional Sobolev inequality with the explicit lower bound.

\begin{theorem}\label{thm2}
For $0<s<\frac{n}{2}$, let $\mathcal{S}_{S}(f)$ denote the Sobolev stability functional given by
\begin{align}\label{sta sob func}
\mathcal{S}_{S}(f)=\frac{S_{n,s}\|(-\Delta)^{s/2}f\|_2^2-\|f\|^2_{2^\ast_s}}{\inf\limits_{h\in M_{s}}\|(-\Delta)^{s/2}(f-h)\|^2_{2}},
\end{align}
then $f\in \dot{H}^s(\mathbb{R}^n)\setminus M_s$, there holds
$$\inf\limits_{f\in \dot{H}^s(\mathbb{R}^n)\setminus M_s}\mathcal{S}_{S}(f)\geq \frac{S_{n,s}\min\{K_{n,s},\min\{2^{\frac{n+2s}{n}}-2,1\}\}}{4}.$$
\end{theorem}

\begin{remark} We note that the explicit lower bound of stability of fractional Sobolev inequality for all $0<s<\frac{n}{2}$ obtained in our paper is not as accurate as the explicit lower bound of the stability of the first order Sobolev inequality obtained by Dolbeault, Esteban, Figalli, Frank and Loss in \cite{DEFFL} in the sense that we do not have the asymptotic behavior when the dimension $n$ goes to infinity.\footnote{More recently, the authors have been able to establish in \cite{CLT2} the asymptotic estimates of the lower bound by proving the lower bound $C_{n, s}$ is at the order of $s$ as $s\to 0$ in all dimensions which was used to prove the global stability of the log-Sobolev inequality on the sphere in \cite{CLT2}, and behaves at the order of $\frac{1}{n}$ as $n\to \infty$ for all $s\in (0, 1)$. }

\end{remark}

\begin{remark}
We also note that K\"{o}nig in \cite{Ko} proved that $$\inf\limits_{f\in \dot{H}^s(\mathbb{R}^n)\setminus M_s}\mathcal{S}_{S}(f)<S_{n,s}\big(\frac{4s}{n+2+2s}\big)$$
by doing the third-order Taylor expansion to fractional Sobolev functional. This shows $\inf\limits_{f\in \dot{H}^s(\mathbb{R}^n)\setminus M_s}\mathcal{S}_{S}(f)$ is strictly smaller than the best constant of local stability of the fractional Sobolev inequalities. For similar results that the best constant for the global stable log-Sobolev inequality is strictly smaller than
the optimal constant of
locally stable log-Sobolev inequality, we refer the reader to \cite{CLT}.
\end{remark}

By the stereographic projection, we can immediately derive the stability of HLS inequality on the sphere $\mathbb{S}^{n}$ with the explicit lower bound.
\begin{corollary}\label{coro1}
For any $g\in L^{\frac{2n}{n+2s}}(\mathbb{S}^n)$, there holds
\begin{equation}\begin{split}\label{HLS-sphere}
&\big(\int_{\mathbb{S}^n}|g(\xi)|^{\frac{2n}{n+2s}}d\sigma_\xi\big)^\frac{n+2s}{n}-B_{n,s}\int_{\mathbb{S}^n}\int_{\mathbb{S}^n}g(\xi)|\xi-\eta|^{-(n-2s)}g(\eta)d\sigma_\xi d\sigma_\eta\\
&\ \ \geq \frac{1}{2}\min\{K_{n,s},\min\{2^{\frac{n+2s}{n}}-2,1\}\} \inf\limits_{h\in \widetilde{M}_{HLS}} \big(\int_{\mathbb{S}^{n}}|g-h|^{\frac{2n}{n+2s}}\big)^{\frac{n+2s}{n}}
\end{split}\end{equation}
where $B_{n,s}=\pi^{-\frac{n-2s}{2}}\big(\frac{\Gamma(\frac{n}{2})}{\Gamma(n)}\big)^{\frac{2s}{n}}\frac{\Gamma(\frac{n}{2}+s)}{\Gamma(s)} $ and $\widetilde{M}_{HLS}$ denotes the space consisting of the extremal functions of sharp HLS inequality on the sphere $\mathbb{S}^n$.

\end{corollary}

Using the sharp Hardy-Littlewood-Sobolev inequality and duality theory, Beckner (see Theorem 5 and Theorem 9 in \cite{Beckner1}) established the following sharp restrictive Sobolev inequalities on the flat sub-manifold $\mathbb{R}^{n-1}$ and the sphere $\mathbb{S}^{n-1}$ respectively:
\begin{equation}\label{res1}
C_{n,s}\int_{\mathbb{R}^n}|(-\Delta)^{\frac{s}{2}}u|^2dx\geq  \big(\int_{\mathbb{R}^{n-1}}|u(x',0)|^{\frac{2(n-1)}{n-2s}}dx'\big)^{\frac{n-2s}{n-1}}
\end{equation}
and
\begin{equation}\label{res2}
D_{n,s}\int_{\mathbb{R}^n}|(-\Delta)^{\frac{s}{2}}u|^2dx\geq \big(\int_{\mathbb{S}^{n-1}}|u(\xi)|^{\frac{2(n-1)}{n-2s}}d\sigma_{\xi}\big)^{\frac{n-2s}{n-1}},
\end{equation}
where $u\in \dot{H}^s(\mathbb{R}^n)$, $s>\frac{1}{2}$ and
$$C_{n,s}=(4\pi)^{-s} \frac{\Gamma(s-\frac{1}{2})}{\Gamma(\frac{n}{2}+s-1)}\frac{\Gamma(\frac{n}{2}-s)}{\Gamma(s)}\big(\frac{\Gamma(n-1)}{\Gamma(\frac{n-1}{2})}\big)^{\frac{2s-1}{n-1}}$$ and $$D_{n,s}=(4\pi)^{-s}\frac{\Gamma(\frac{n-2s}{2})}{\Gamma(\frac{n-2+2s}{2})}\big(\frac{\Gamma(n-1)}{\Gamma(\frac{n-1}{2})}\big)^{\frac{2s-1}{n-1}}\frac{\Gamma(s-\frac{1}{2})}{\Gamma(s)}$$ are the best constants of restrictive Sobolev inequalities \eqref{res1} and \eqref{res2} respectively.  From  Beckner's proof, we can see that the restrictive Sobolev inequality \eqref{res1} is equivalent to the HLS inequality in $\mathbb{R}^{n-1}$. Hence we can deduce that equality in inequality \eqref{res1} holds if and only if $u=(-\Delta)^{-\frac{s}{2}}(T^{*}(g))$, where $g$ is the extremal function of HLS inequality in $\mathbb{R}^{n-1}$ and $$T^{*}(g)(x', x_n)=c_{n,s}\int_{\mathbb{R}^{n-1}}\frac{g(y')}{\big(|x'-y'|^2+|x_n|^2\big)^{\frac{n-s}{2}}}dy'$$ (see also the proof of Theorem \ref{thm3}). Similarly, the equality in the restrictive Sobolev inequality \eqref{res2} holds if and only if $u=(-\Delta)^{-\frac{s}{2}}(\widetilde{T}^{*}(g))$, where $g$ is the extremal function of HLS inequality in $\mathbb{S}^{n-1}$ and $$\widetilde{T}^{*}(g)(x)=c_{n,s}\int_{\mathbb{S}^{n-1}}\frac{g(\xi)}{|x-\xi|^{n-s}}d\sigma_{\xi}$$ (see also the proof of Theorem \ref{thm4}).
\medskip

 As an application of Theorem \ref{thm1}, we will further establish the stability of sharp restrictive Sobolev inequality on the flat sub-manifold $\mathbb{R}^{n-1}$ and the sphere $\mathbb{S}^{n-1}$ with the explicit lower bound.

\begin{theorem}\label{thm3}
For any $f\in \dot{H}^s(\mathbb{R}^n)\setminus M_{\mathbb{R}^{n-1}}$, there holds
\begin{equation*}\begin{split}
&C_{n,s}\int_{\mathbb{R}^n}|(-\Delta)^{\frac{s}{2}}f|^2dx- \big(\int_{\mathbb{R}^{n-1}}|f(x',0)|^{\frac{2(n-1)}{n-2s}}dx'\big)^{\frac{n-2s}{n-1}}\\
&\ \ \geq \frac{C_{n,s}\min\{K_{n-1,s-\frac{1}{2}},\min\{2^{\frac{n-2+2s}{n-1}}-2,1\}\}}{4}\inf\limits_{h\in M_{\mathbb{R}^{n-1}}}\int_{\mathbb{R}^n}|(-\Delta)^{\frac{s}{2}}(f-h)|^2dx,
\end{split}\end{equation*}
where $M_{\mathbb{R}^{n-1}}$ denotes the space consisting of the extremal functions of sharp restrictive Sobolev inequality on flat sub-manifold $\mathbb{R}^{n-1}$ \eqref{res1}.
\end{theorem}
\begin{remark}
The sharp  Sobolev trace inequality for first order derivative and extremals have been already obtained by Escobar in \cite{Escobar}. However, the stability of  Sobolev trace inequality for first order derivative remains unknown. Our Theorem \ref{thm3} in the case of $s=1$ in fact gives the stability of trace Sobolev inequality for first order derivative with the explicit lower bound. We also note that the sharp  Sobolev trace inequality involving the norm $\|\nabla u\|_{L^p(\mathbb{R}^{n}_{+})}$ and extremal have been obtained by Nazaret in \cite{Nazaret}.

\end{remark}

\begin{theorem}\label{thm4}
For any $f\in \dot{H}^s(\mathbb{R}^n)\setminus M_{\mathbb{S}^{n-1}}$, there holds
\begin{equation*}\begin{split}
&D_{n,s}\int_{\mathbb{R}^n}|(-\Delta)^{\frac{s}{2}}f|^2dx- \big(\int_{\mathbb{S}^{n-1}}|f(\xi)|^{\frac{2(n-1)}{n-2s}}d\sigma_{\xi}\big)^{\frac{n-2s}{n-1}}\\
&\ \ \geq \frac{D_{n,s}\min\{K_{n-1,s-\frac{1}{2}},\min\{2^{\frac{n-2+2s}{n-1}}-2,1\}\}}{4}\inf_{h\in M_{\mathbb{S}^{n-1}}} \int_{\mathbb{R}^n}|(-\Delta)^{\frac{s}{2}}(f-h)|^2dx,
\end{split}\end{equation*}
where $M_{\mathbb{S}^{n-1}}$ denotes the space consisting of the extremal functions of sharp restrictive Sobolev inequality on the sphere $\mathbb{S}^{n-1}$ \eqref{res2}.
\end{theorem}
This paper is organized as follows. Section 2 is devoted to some preliminaries. In Section 3, we will give the stability of HLS inequality with explicit lower bounds using competing symmetries, the continuous Steiner symmetrization inequality for the HLS integral and the local dual stability theory. In Section 4, we will establish the stability of fractional Sobolev inequality with explicit lower bounds. Sections 5 and 6 are devoted to calculating the lower-bound of stability of sharp restrictive Sobolev inequality on the flat sub-manifold $\mathbb{R}^{n-1}$ and the sphere $\mathbb{S}^{n-1}$.

\section{Preliminaries}
In this section, we will state some tools, including competing symmetry theorem, a continuous rearrangement flow and expansions with remainder term and so on,
which will be used in our proof.
\medskip

Let $f\in L^{p}(\mathbb{R}^n)$, $1<p<\infty$ and $f_k=(\mathcal{R}U)^kf,~~k\in \mathbb{N}$, where $\mathcal{R}f=f^\ast$ is the decreasing rearrangement of $f$ and
$$(Uf)(x)=\left(\frac{2}{|x-e_n|^2}\right)^{\frac{n}{p}}f\left(\frac{x_1}{|x-e_n|^2},\cdots,\frac{x_{n-1}}{|x-e_n|^2},\frac{|x|^2-1}{|x-e_n|^2}\right),$$
$e_n=(0,\cdots,0,1)\in \mathbb{R}^n$. Carlen and Loss~\cite{CaLo} proved the following competing symmetry theorem.

\begin{theorem}\label{compete sym}(Carlen-Loss)
Let $0\leq f\in L^{p}$ ($1<p<\infty$), $f_k=(\mathcal{R}U)^kf$ and $h(x)=\|f\|_{p}|\mathbb{S}^n|^{-1/p}\left(\frac{2}{1+|x|^2}\right)^{n/p}$.  Then
$$\lim_{k\rightarrow \infty}\|f_k(x)-h(x)\|_{p}=0.$$
\end{theorem}

\vskip0.3cm

A continuous rearrangement flow which interpolates between a function and its symmetric decreasing rearrangement introduced in \cite{DEFFL} based on Brock's flow (\cite{Br1}, \cite{Br2}) plays an important part in our proof. More specifically, there
exists a flow $f_\tau$, $\tau\in [0,\infty]$, such that
$$f_0=f,~~f_\infty=f^\ast.$$
And if $0\leq f\in L^p(\mathbb{R}^n)$ for some $1\leq p<\infty$, then $\tau\rightarrow f_{\tau}$ is continuous in $L^{p}(\mathbb{R}^n)$. In our setting we need to
prove $\tau\rightarrow \inf\limits_{h\in M_{HLS}}\|f_\tau-h\|_{\frac{2n}{n+2s}}$ is continuous.

\begin{lemma}\label{continous}
Let $0\leq f\in L^{\frac{2n}{n+2s}}$. Then the function $\tau\rightarrow \inf\limits_{h\in M_{HLS}}\|f_\tau-h\|_{\frac{2n}{n+2s}}$ is continuous.
\end{lemma}

\begin{proof}
Let us prove $\inf\limits_{h\in M_{HLS}}\|f_\tau-h\|_{\frac{2n}{n+2s}}\rightarrow\inf\limits_{h\in M_{HLS}}\|f_{\tau_0}-h\|_{\frac{2n}{n+2s}}$ as $\tau\rightarrow \tau_0$.
For any $\varepsilon>0$, there exists a $g_{1,\tau}\in M_{HLS}$ such that
$$\inf_{h\in M_{HLS}}\|f_\tau-h\|_{\frac{2n}{n+2s}}\geq \|f_\tau-g_{1,\tau}\|_{\frac{2n}{n+2s}}-\varepsilon.$$
Then
\begin{align}\label{upp est}
\inf\limits_{h\in M_{HLS}}\|f_\tau-h\|_{\frac{2n}{n+2s}}-\inf\limits_{h\in M_{HLS}}\|f_{\tau_0}-h\|_{\frac{2n}{n+2s}}\geq
 \|f_\tau-g_{1,\tau}\|_{\frac{2n}{n+2s}}- \|f_{\tau_0}-g_{1,\tau}\|_{\frac{2n}{n+2s}}-\varepsilon.
 \end{align}
Applying the triangle inequality, we obtain that $$\inf\limits_{h\in M_{HLS}}\|f_\tau-h\|_{\frac{2n}{n+2s}}-\inf\limits_{h\in M_{HLS}}\|f_{\tau_0}-h\|_{\frac{2n}{n+2s}}\geq -\|f_{\tau}-f_{\tau_0}\|_{\frac{2n}{n+2s}}-\varepsilon.$$
On the other hand, there exist a $g_0\in M_{HLS}$ such that
$$\inf_{h\in M_{HLS}}\|f_{\tau_0}-h\|_{\frac{2n}{n+2s}}\geq \|f_{\tau_0}-g_0\|_{\frac{2n}{n+2s}}-\varepsilon.$$
Thus we can similarly derive that
\begin{align*}\label{low est}
&\inf\limits_{h\in M_{HLS}}\|f_\tau-h\|_{\frac{2n}{n+2s}}-\inf\limits_{h\in M_{HLS}}\|f_{\tau_0}-h\|_{\frac{2n}{n+2s}}\\
&\ \ \leq \|f_\tau-g_0\|_{\frac{2n}{n+2s}}- \|f_{\tau_0}-g_0\|_{\frac{2n}{n+2s}}+\varepsilon\\
&\ \ \leq \|f_{\tau}-f_{\tau_0}\|_{\frac{2n}{n+2s}}+\epsilon.
\end{align*}
Combining the above estimate, we obtain for any $\epsilon>0$, there holds
\begin{equation*}\begin{split}
|\inf\limits_{h\in M_{HLS}}\|f_\tau-h\|_{\frac{2n}{n+2s}}-\inf\limits_{h\in M_{HLS}}\|f_{\tau_0}-h\|_{\frac{2n}{n+2s}}|\leq \|f_\tau-f_{\tau_0}\|_{\frac{2n}{n+2s}}+\epsilon.
\end{split}\end{equation*}

Applying the continuity of $\tau\rightarrow f_{\tau}$ in $L^{\frac{2n}{n+2s}}(\mathbb{R}^n)$ and let $\epsilon\rightarrow 0$, we conclude that $$\lim\limits_{\tau\rightarrow \tau_0}\inf\limits_{h\in M_{HLS}}\|f_\tau-h\|_{\frac{2n}{n+2s}}=\inf\limits_{h\in M_{HLS}}\|f_{\tau_0}-h\|_{\frac{2n}{n+2s}}.$$ Then the proof of Lemma \ref{continous} is accomplished.
\end{proof}

The following Taylor expansions of $\|u+r\|^2_{2^\ast_s}$ are also needed in our proof. We also note that when $s=1$, Taylor expansions of $\|u+r\|^2_{2^\ast_s}$ has been obtained in \cite{DEFFL}.
\begin{lemma}\label{expansion}
Let $u,r\in L^{2^\ast_s}(\mathbb{R}^n)$, $u+r\geq0$ and $u\geq 0$.
If $2^\ast_s$ is not an integer, then
\begin{equation*}\begin{split}
\|u+r\|^2_{2^\ast_s}&\leq \|u\|^2_{2^\ast_s}+\frac{2}{2^\ast_s}\sum_{k=1}^{[2^\ast_s]}\frac{2^{\ast}_s(2^\ast_s-1)\cdots(2^\ast_s-k+1)}{k!}\|u\|_{2^\ast_s}^{2-2^\ast_s}
\int_{\mathbb{R}^n}u^{2^\ast_s-k}r^kdx\\
&\ \ +\frac{2}{2^\ast_s}\|u\|_{2^\ast_s}^{2-2^\ast_s}\|r\|^{2^\ast_s}_{2^\ast_s},
\end{split}\end{equation*}
where $[2^\ast_s]$ is the integer part of $2^\ast_s$. If $2^\ast_s$ is an integer, then
$$\|u+r\|^2_{2^\ast_s}\leq \|u\|^2_{2^{\ast}_s}+\frac{2}{2^\ast_s}\sum_{k=1}^{[2^\ast_s]}\frac{2^{\ast}_s(2^\ast_s-1)\cdots(2^\ast_s-k+1)}{k!}\|u\|_{2^\ast_s}^{2-2^\ast_s}
\int_{\mathbb{R}^n}u^{2^\ast_s-k}r^kdx.$$
\end{lemma}
In order to prove Lemma ~\ref{expansion}, we need the following technical lemma.
\begin{lemma}
(1) For all $x\geq -1$, $q\geq 1$, $q$ is not an integer,
$$(1+x)^{q}\leq 1+\sum_{k=1}^{[q]}\frac{q(q-1)\cdots(q-k+1)}{k!}x^k+|x|^q,$$
where $[q]$ is the integer part of $q$.
\vskip0.1cm

(2) For all $x\geq -1$, $q\geq 1$, $q$ is an integer,
$$(1+x)^{q}= 1+\sum_{k=1}^{q}\frac{q(q-1)\cdots(q-k+1)}{k!}x^k.$$
\end{lemma}

\begin {proof}
We only need to consider the case when $q$ is not an integer. Let
$$f(x)=(1+x)^q-1-\sum_{k=1}^{[q]}\frac{q(q-1)\cdots(q-k+1)}{k!}x^k-|x|^q.$$
We first prove that $f(x)\leq 0$ when $x\geq 0$. Since
$$f^{\prime}(x)=q(1+x)^{q-1}-\sum_{k=1}^{[q]}\frac{q(q-1)\cdots(q-k+1)}{(k-1)!}x^{k-1}-q x^{q-1},$$
$$\cdot\cdot\cdot,$$
$$f^{([q])}(x)=q(q-1)\cdots(q-[q]+1)\big((1+x)^{q-[q]}-1-x^{q-[q]}\big),$$
then
$$f(0)=f^{\prime}(0)=\cdots=f^{([q])}(0)=0.$$
Using the fundamental inequality $(1+x)^{\alpha}\leq 1+x^\alpha$ for $0<\alpha<1$ and $x\geq 0$, we obtain
$f^{([q])}(x)\leq 0~(x>0)$. Then $f^{([q]-1)}(x)$ is decreasing on $(0,+\infty)$, which implies
$f^{([q]-1)}(x)\leq f^{([q]-1)}(0)=0$. Repeat the deduction and we can obtain $f(x)\leq 0$ when $x\geq 0$.
\vskip0.1cm

\medskip

When $-1\leq x\leq 0$, we have
$$f^{\prime}(x)=q(1+x)^{q-1}-\sum_{k=1}^{[q]}\frac{q(q-1)\cdots(q-k+1)}{k!}x^{k-1}+q(-x)^{q-1},$$
$$\cdots,$$
$$f^{([q])}(x)=q(q-1)\cdots(q-[q]+1)\left\{(1+x)^{q-[q]}-1-(-1)^{[q]}(-x)^{q-[q]}\right\},$$
and there still holds
$$f(0)=f^{\prime}(0)=\cdots=f^{([q])}(0)=0.$$
If $[q]$ is even, using the fundamental inequality
$$(1+x)^\alpha\leq 1+(-x)^\alpha,~~0<\alpha<1,~~-1\leq x\leq 0$$
we know $f^{([q])}(x)\leq 0$. Then $f^{([q]-1)}(x)$ is decreasing on $[-1,0]$, which implies $f^{([q]-1)}(x)\geq f^{([q]-1)}(0)=0$.
Thus $f^{([q]-2)}(x)$ is increasing on $[-1,0]$, which means $f^{([q]-2)}(x)\leq f^{([q]-2)}(0)=0$. Repeat the deduction and we can obtain
$f^{(2k)}(x)\leq 0, f^{(2k+1)}(x)\geq 0$ when $k=0,1,\cdots,\frac{[q]}{2}$.

\medskip

Now let us deal with the case when $[q]$ is odd. Since the function
$$h(x)=(1+x)^\alpha-1+(-x)^{\alpha}~~(0<\alpha<0)$$ satisfy that
$h^{\prime}(x)=\alpha[(1+x)^{\alpha-1}-(-x)^{\alpha-1}]$ is nonnegative on $[-1,-1/2)$ and nonpositive on $(-1/2,0]$, then $f^{([q])}(x)$ is increasing on $[-1,-1/2)$ and decreasing on $(-1/2,0]$, which implies $f^{([q])}(x)\geq f^{([q])}(0)=f^{([q])}(-1)=0$. Thus $f^{([q]-1)}(x)$ is increasing on $[-1,0]$
which means $f^{([q]-1)}(x)\leq f^{([q]-1)}(0)=0,~~x\in[-1,0]$. Then $f^{([q]-2)}(x)$ is decreasing on $[-1,0]$, which means $f^{([q]-2)}(x)\geq f^{[q]-2}(0)=0$. Repeat the deduction and we can obtain
$f^{(2k)}(x)\leq 0, f^{(2k+1)}(x)\geq 0$ when $k=0,1,\cdots,\frac{[q]-1}{2}$, which complete the proof.
\end{proof}

Now let us prove Lemma ~\ref{expansion}.
Here we only state the proof when $2^\ast_s$ is not an integer. By Proposition 1, we derive that
$$(u+r)^{2^\ast_s}\leq u^{2^\ast_s}+\sum_{k=1}^{[2^\ast_s]}\frac{2^{\ast}_s(2^\ast_s-1)\cdots(2^{\ast}_s-k+1)}{k!}u^{2^\ast_s-k}r^k+|r|^{2^\ast_s},$$
thus
$$\int_{\mathbb{R}^n}(u+r)^{2^\ast_s}dx\leq \int_{\mathbb{R}^n}u^{2^\ast_s}dx+\sum_{k=1}^{[2^\ast_s]}\frac{2^\ast_s(2^\ast_s-1)\cdots(2^\ast_s-k+1)}{k!}\int_{\mathbb{R}^n}u^{2^\ast_s-k}r^kdx
\int_{\mathbb{R}^n}|r|^{2^\ast_s}dx.$$
Using the inequality $(1+x)^{\alpha}\leq 1+\alpha x$ for $x\geq -1$ and $0<\alpha<1$, we complete the proof of Lemma \ref{expansion}.

\section{Stability of the Hardy-Littlewood-Sobolev inequality}
In this section, we will set up the stability of the Hardy-Littlewood-Sobolev inequality with explicit constant. First we will establish the
local version of the stability of the fractional Sobolev inequality for nonnegative functions. Then, using the dual method from \cite{Car} by Carlen, we can deduce the local stability for the HLS inequality for nonnegative function (where the aim function is close to the manifold of the HLS optimizers). When the aim function is positive and away from the manifold of HLS optimizers, we will handle them by the competing symmetries and Block's flow. At last, we will deal with the stability of HLS inequality when the aim function is not necessarily non-negative.

\subsection{Local stability of fractional Sobolev inequality}
In the subsection we will set up the local stability of fractional Sobolev inequality with explicit lower bounds (Lemma~\ref{local Sobolev stability}) for $0<s<\frac{n}{2}$, while the case $s=1$ was proved in
\cite{DEFFL}. Without loss of generality, we only give the proof when $2^{*}_s$ is not an integer.

\vskip0.5cm

Let $0\leq f\in \dot{H}^s(\mathbb{R}^n)$.
It is well known that there exists an $0\leq u\in M_{S}$ such that $\|(-\Delta)^{s/2}(f-u)\|_2^2=\inf\limits_{h\in M_S}\|(-\Delta)^{s/2}(f-h)\|_2^2.$ Then  $r=f-u$ satisfies
$$\int_{\mathbb{R}^n}[(-\Delta)^{s/2}u][(-\Delta)^{s/2}r]dx=0,$$ and
$$\int_{\mathbb{R}^n}u^{2^\ast_s-1}rdx=0,$$
since $(-\Delta)^{s}u=c_u u^{2^\ast_s-1}$ with $c_u=\frac{\|(-\Delta)^{s/2}u\|_2^2}{\|u\|_{2^\ast_s}^{2}}$.
Then
\begin{equation}\label{Sob deficit}
\|(-\Delta)^{s/2}f\|_2^2-S^{-1}_{n,s}\|f\|^2_{2^{\ast}_s}=\|(-\Delta)^{s/2}u\|_2^2+\|(-\Delta)^{s/2}r\|^2_2-S^{-1}_{n,s}\|u+r\|^{2}_{2^\ast_s}.
\end{equation}

Next we will estimate the Sobolev deficit $\|(-\Delta)^{s/2}f\|_2^2-S_{n,s}\|f\|^2_{2^{\ast}_s}$ by the expansion from Lemma \ref{expansion} and the following spectral gap inequalities which are due to Rey~\cite{Re} (Eq. (D.1)) and Esposito~\cite{Es} (Lemma 2.1) for $s=1$, to De Nitti and K\"{o}nig ~\cite{DK} (Prop. 3.4)
for $0<s<n/2$.
\begin{lemma}\label{spectral gap inequality}
Given $f\in \dot{H}^s(\mathbb{R}^n)$. Let $u\in M_{s}$ such that
$$\|(-\Delta)^{s/2}(f-u)\|_2^2=\inf\limits_{h\in M_s}\|(-\Delta)^{s/2}(f-h)\|_2^2.$$
Then $r=f-u$ satisfies
$$\|(-\Delta)^{s/2}r\|_2^2-(2^\ast_s-1)S^{-1}_{n,s}\|u\|_{2^\ast_s}^{2-2^\ast_s}\int_{\mathbb{R}^n}u^{2^\ast_s-2}r^2dx\geq \frac{4s}{n+2s+2}\|(-\Delta)^{s/2}r\|_2^2.$$
\end{lemma}

\vskip0.3cm

Now by (\ref{Sob deficit}) and Lemma \ref{expansion},  we have
\begin{align}\label{est of sob def}\nonumber
&\|(-\Delta)^{s/2}f\|_2^2-S^{-1}_{n,s}\|f\|^2_{2^{\ast}_s}\geq\|(-\Delta)^{s/2}r\|_2^2-(2^\ast_s-1)S_{n,s}\|u\|_{2^\ast_s}^{2-2^\ast_s}\int_{\mathbb{R}^n}u^{2^\ast_s-2}r^2dx\\
& -S_{n,s}\left\{\frac{2}{2^\ast_s}\sum_{k=3}^{[2^\ast_s]}\frac{2^{\ast}_s(2^\ast_s-1)\cdots(2^\ast_s-k+1)}{k!}\|u\|_{2^\ast_s}^{2-2^\ast_s}
\int_{\mathbb{R}^n}u^{2^\ast_s-k}r^kdx+\frac{2}{2^\ast_s}\|u\|_{2^\ast_s}^{2-2^\ast_s}\|r\|^{2^\ast_s}_{2^\ast_s}\right\}.
\end{align}
By H\"{o}lder's inequality and sharp Sobolev inequality, there holds
\begin{align*}
\|u\|_{2^\ast_s}^{2-2^\ast_s}\int_{\mathbb{R}^n}u^{2^\ast_s-k}r^kdx\leq \|u\|_{2^\ast_s}^{2-k}\|r\|^k_{2^\ast_s}
\leq S_{n,s}\left(\frac{\|r\|_{2^\ast_s}}{\|u\|_{2^\ast_s}}\right)^{k-2}\|(-\Delta)^{s/2}r\|_{2}^2,
\end{align*}
and
$$\|u\|_{2^\ast_s}^{2-2^\ast_s}\|r\|^{2^\ast_s}_{2^\ast_s}\leq S_{n,s}\left(\frac{\|r\|_{2^\ast_s}}{\|u\|_{2^\ast_s}}\right)^{2_s^\ast-2}\|(-\Delta)^{s/2}r\|_{2}^2.$$
Therefore combining Lemma \ref{spectral gap inequality}, (\ref{est of sob def}) and the above estimates, we get
\begin{align}\label{est of sta func}\nonumber
& \frac{\|(-\Delta)^{s/2}f\|_2^2-S^{-1}_{n,s}\|f\|^2_{2^{\ast}_s}}{\|(-\Delta)^{s/2}r\|_{2}^2}\\
& \geq\left\{\frac{4s}{n+2s+2}-
\frac{2}{2^\ast_s}\sum_{k=3}^{[2^\ast_s]}\frac{2^{\ast}_s(2^\ast_s-1)\cdots(2^\ast_s-k+1)}{k!}\left(\frac{\|r\|_{2^\ast_s}}{\|u\|_{2^\ast_s}}\right)^{k-2}-
\frac{2}{2^\ast_s}\left(\frac{\|r\|_{2^\ast_s}}{\|u\|_{2^\ast_s}}\right)^{2^\ast_s-2}\right\}.
\end{align}

Now we are in the position to prove the local stability of Sobolev inequality for nonnegative functions. Let  $l(\delta)=\sqrt{\frac{\delta}{1-\delta}}$ and
recall
$$\nu(\delta)=\inf\left\{\mathcal{S}_{S}(f):0\leq f\in \dot{H}^s(\mathbb{R}^n)\setminus M_s,\inf_{g\in M_s}\|(-\Delta)^{s/2}(f-g)\|_2^2\leq \delta\|(-\Delta)^{s/2}f\|_2^2\right\},$$
and
$$m(\delta)=\frac{4s}{n+2s+2}-
\frac{2}{2^\ast_s}\sum_{k=3}^{2^\ast_s}\frac{2^{\ast}_s(2^\ast_s-1)\cdots(2^\ast_s-k+1)}{k!}l(\delta)^{k-2},~~\text{when}~~2^\ast_s~~\text{is an integer}, $$
otherwise
$$m(\delta)=\frac{4s}{n+2s+2}-
\frac{2}{2^\ast}\sum_{k=3}^{[2^\ast_s]}\frac{2^{\ast}_s(2^\ast_s-1)\cdots(2^\ast-k+1)}{k!}l(\delta)^{k-2}-
\frac{2}{2^\ast_s}l(\delta)^{2^\ast_s-2}.$$
\begin{lemma}\label{local Sobolev stability}
With the above notations, we have $\nu(\delta)\geq m(\delta)$.
\end{lemma}
\begin{proof}
Since $\|(-\Delta)^{s/2}f\|_2^2=\|(-\Delta)^{s/2}r\|_2^2+\|(-\Delta)^{s/2}u\|_2^2$, then
\begin{align}\label{est}
\frac{\|r\|_{2^\ast_s}}{\|u\|_{2^\ast_s}}\leq \frac{\|(-\Delta)^{s/2}r\|_2}{\|(-\Delta)^{s/2}u\|_2}=\frac{\|(-\Delta)^{s/2}r\|_2}{\|(-\Delta)^{s/2}f\|_2}
\frac{1}{\sqrt{1-\frac{\|(-\Delta)^{s/2}r\|^2_2}{\|(-\Delta)^{s/2}f\|^2_2}}}\leq \sqrt{\frac{\delta}{1-\delta}}=l(\delta).
\end{align}
Therefor by (\ref{est of sta func}) and (\ref{est})
$$\nu(\delta)\geq S_{n,s}\left\{\frac{4s}{n+2s+2}-
\frac{2}{2^\ast_s}\sum_{k=3}^{[2^\ast_s]}\frac{2^{\ast}_s(2^\ast_s-1)\cdots(2^\ast_s-k+1)}{k!}l(\delta)^{k-2}-
\frac{2}{2^\ast_s}l(\delta)^{2^\ast_s-2}\right\},$$
which completes the proof of local stability of fractional Sobolev inequality with explicit lower bounds.
\end{proof}

\subsection{Stability of HLS inequality for positive functions} In this subsection, we will consider the stability of HLS inequality for
positive functions. Choose $\delta$ small enough, and recall
$$\mu(\delta)=\inf\{\mathcal{S}_{HLS}(g):0\leq g \in L^{\frac{2n}{n+2s}}(\mathbb{R}^n)\setminus M_{HLS}, \inf_{h\in M_{HLS}}\|g-h\|^2_{\frac{2n}{n+2s}}\leq \delta\|g\|^2_{\frac{2n}{n+2s}}\}.$$
First we obtain the local stability of HLS inequality by the following dual lemma from Carlen~\cite{Car} and the local stability
of Sobolev inequality for nonnegative functions.
\begin{lemma}\label{local}
If $\mathcal{S}_{S}(f)\geq m(\delta)$ for all nonnegative function $f$ with
$$\inf\limits_{h\in M_s}\|(-\Delta)^{s/2}(f-h)\|^2_2\leq \delta\|(-\Delta)^{s/2}f\|_2^2,$$
then
$$\mathcal{S}_{HLS}(g)\geq \frac{1}{2}\frac{n-2s}{n+2s}\min\{m(\delta)\frac{n-2s}{n+2s},1\},$$
for all $0<g \in L^{\frac{2n}{n+2s}}(\mathbb{R}^n)\setminus M_{HLS}$ satisfying $\inf\limits_{h\in M_{HLS}}\|g-h\|^2_{\frac{2n}{n+2s}}\leq \frac{\delta}{2}\|g\|^2_{\frac{2n}{n+2s}}$.
\end{lemma}
\begin{proof}
If $\|g\|^2_{\frac{2n}{n+2s}}\geq 2S_{n,s}^{-1}\|(-\Delta)^{-s/2}g\|_2^2$, then
$$\|g\|^2_{\frac{2n}{n+2s}}-S_{n,s}^{-1}\|(-\Delta)^{-s/2}g\|_2^2\geq \frac{1}{2}\|g\|^2_{\frac{2n}{n+2s}}\geq \frac{1}{2}\inf\limits_{h\in M_{HLS}}\|g-h\|^2_{\frac{2n}{n+2s}},$$
thus $\mathcal{S}_{HLS}(g)\geq \frac{1}{2}$. When $\|g\|^2_{\frac{2n}{n+2s}}\leq 2S_{n,s}^{-1}\|(-\Delta)^{-s/2}g\|_2^2$, using Lemma 3.4 in \cite{Car},
we can obtain $\mathcal{S}_{HLS}(g)\geq \frac{1}{2}\frac{n-2s}{n+2s}\min\{m(\delta)\frac{n-2s}{n+2s},1\}.$ Since
$$\frac{1}{2}\frac{n-2s}{n+2s}\min\{m(\delta)\frac{n-2s}{n+2s},1\}\leq \frac{1}{2},$$ then
$$\mathcal{S}_{HLS}(g)\geq \frac{1}{2}\frac{n-2s}{n+2s}\min\{m(\delta)\frac{n-2s}{n+2s},1\}$$ if
$\inf\limits_{h\in M_{HLS}}\|g-h\|^2_{\frac{2n}{n+2s}}\leq \frac{\delta}{2}\|g\|^2_{\frac{2n}{n+2s}}$. This accomplishes the proof of Lemma \ref{local}.

\end{proof}

Next, let us handle the stability of HLS inequality when the nonnegative function $g$ satisfies
$\inf\limits_{h\in M_{HLS}}\|g-h\|^2_{\frac{2n}{n+2s}}> \delta\|g\|^2_{\frac{2n}{n+2s}}$. We emphasize here that the idea of the proof is much inspired by the work \cite{DEFFL}. Dolbeault, Esteban, Figalli, Frank and Loss in \cite{DEFFL} applied the technique of rearrangement flow for $L^2$ integral of gradient $u$ to establish the relation between local stability and global stability of Sobolev inequality for the first order derivative.
However, such a relationship cannot be extended to the higher and fractional order of derivatives.  To this end, we will establish the relation between local stability and global stability of HLS inequality by using the technique of the rearrangement flow for HLS integral.
\begin{lemma}\label{away}
For any $0\leq g \in L^{\frac{2n}{n+2s}}(\mathbb{R}^n)$ satisfying $\inf\limits_{h\in M_{HLS}}\|g-h\|^2_{\frac{2n}{n+2s}}> \delta\|g\|^2_{\frac{2n}{n+2s}}$, there holds
$$\mathcal{S}_{HLS}(g)\geq \delta \mu(\delta).$$
\end{lemma}

\begin{proof}
Assume that $0\leq g \in L^{\frac{2n}{n+2s}}(\mathbb{R}^n)$ satisfying $\inf\limits_{h\in M}\|g-h\|^2_{\frac{2n}{n+2s}}> \delta\|g\|^2_{\frac{2n}{n+2s}}$. Let $g_k=(\mathcal{R}U)^kg$ (see the definition of $\mathcal{R}U$ in section 2). By Theorem \ref{compete sym} we know
$$\|g_k-h_g\|_{\frac{2n}{n+2s}}\rightarrow 0,~~k\rightarrow \infty,$$
where $h_g=\|g\|_{\frac{2n}{n+2s}}|\mathbb{S}^n|^{-\frac{n-2s}{2n}}(\frac{2}{1+|x|^2})^{\frac{n+2s}{2}}\in M_{HLS}$. It is well known that $$k\rightarrow\|(-\Delta)^{-\frac{s}{2}}g_k\|^2_2$$
is increasing by Riesz rearrangement inequality and $\|g_k\|_{\frac{2n}{n+2s}}=\|g\|_{\frac{2n}{n+2s}}$.
Thus
\begin{align}\label{est of hls}\nonumber
& \mathcal{S}_{HLS}(g)\geq \frac{\|g\|^2_{\frac{2n}{n+2s}}-S_{n,s}^{-1}\|(-\Delta)^{-s/2}g\|_2^2}{\|g\|^2_{\frac{2n}{n+2s}}}\\
&\geq  1-\frac{S_{n,s}^{-1}\|(-\Delta)^{-s/2}g\|_2^2}{\|g\|^2_{\frac{2n}{n+2s}}}\geq \frac{\|g_k\|^2_{\frac{2n}{n+2s}}-S_{n,s}^{-1}\|(-\Delta)^{-s/2}g_k\|_2^2}{\|g_k\|^2_{\frac{2n}{n+2s}}}.
\end{align}
Since $\|g_k-h_g\|_{\frac{2n}{n+2s}}\rightarrow 0$ as $k\rightarrow \infty$ and $h_g\in M_{HLS}$, then there exist a $k_0\in \mathbb{N}$ such that
 $$\inf\limits_{h\in M_{HLS}}\|g_{k_0}-h\|^2_{\frac{2n}{n+2s}}>\delta\|g_{k_0}\|^2_{\frac{2n}{n+2s}}$$
  and
 $$\inf\limits_{h\in M_{HLS}}\|g_{k_0+1}-h\|^2_{\frac{2n}{n+2s}}\leq \delta\|g_{k_0+1}\|^2_{\frac{2n}{n+2s}}.$$
 Denote $g_0=Ug_{k_0}$, $g_\infty=g_{k_0+1}$, then
 $$\inf\limits_{h\in M_{HLS}}\|g_{0}-h\|^2_{\frac{2n}{n+2s}}=\inf\limits_{h\in M_{HLS}}\|g_{k_0}-h\|^2_{\frac{2n}{n+2s}}
 >\delta\|g_{k_0}\|^2_{\frac{2n}{n+2s}}=\delta\|g_{0}\|^2_{\frac{2n}{n+2s}}.$$
 Now using the continuous rearrangement flow $g_\tau$ ($0\leq\tau\leq\infty$) introduced in section 2, we conclude that $g_\tau$
 satisfies
$$\|(-\Delta)^{-\frac{s}{2}}g_\tau\|_{2}^2\geq \|(-\Delta)^{-\frac{s}{2}}g_0\|_{2}^2=\|(-\Delta)^{-\frac{s}{2}}g_{k_0}\|_{2}^2, ~\text{and}~\|g_\tau\|_{\frac{2n}{n+2s}}=\|g_0\|_{\frac{2n}{n+2s}}=\|g\|_{\frac{2n}{n+2s}}.$$
Since $\tau\rightarrow\inf\limits_{h\in M_{HLS}}\|g_\tau-h\|^2_{\frac{2n}{n+2s}}$ is continuous by lemma \ref{continous}, then there exists $\tau_0\in (0,\infty)$ such that
\begin{align}\label{equality}
\inf\limits_{h\in M_{HLS}}\|g_{\tau_0}-h\|^2_{\frac{2n}{n+2s}}=\delta\|g_{\tau_0}\|^2_{\frac{2n}{n+2s}}.
\end{align}
Therefore by (\ref{est of hls}) and (\ref{equality}),
\begin{align*}
\mathcal{S}_{HLS}(g)&\geq \frac{\|g_0\|^2_{\frac{2n}{n+2s}}-S_{n,s}^{-1}\|(-\Delta)^{-s/2}g_0\|_2^2}{\|g_0\|^2_{\frac{2n}{n+2s}}}\\
&\geq \frac{\|g_{\tau_0}\|^2_{\frac{2n}{n+2s}}-S_{n,s}^{-1}\|(-\Delta)^{-s/2}g_{\tau_0}\|_2^2}{\|g_{\tau_0}\|^2_{\frac{2n}{n+2s}}}\\
&=\delta \frac{\|g_{\tau_0}\|^2_{\frac{2n}{n+2s}}-S_{n,s}^{-1}\|(-\Delta)^{-s/2}g_{\tau_0}\|_2^2}{\inf\limits_{h\in M_{HLS}}\|g_{\tau_0}-h\|^2_{\frac{2n}{n+2s}}}
\geq \delta\mu(\delta),
\end{align*}
where in the second inequality we use the following Riesz's inequality for continuous convex rearrangement (see Appendix A in \cite{DEFFL}):
For nonnegative functions $f,g$, there holds $$\iint_{\mathbb{R}^n\times \mathbb{R}^n}\frac{f_\tau(x)g_\tau(y)}{|x-y|^{n-2s}}dxdy\geq \iint_{\mathbb{R}^n\times \mathbb{R}^n}\frac{f(x)g(y)}{|x-y|^{n-2s}}dxdy.$$
This proves Lemma \ref{away}.
\end{proof}

According to the definition of $\mu(\delta)$, through Lemma \ref{local}, we deduce that $$\mu(\delta)\geq \frac{1}{2}\frac{n-2s}{n+2s}\min\{m(2\delta)\frac{n-2s}{n+2s},1\}.$$

Combining this and Lemma \ref{away}, we derive that for any $0\leq f\in L^{\frac{2n}{n+2s}}(\mathbb{R}^n)\setminus M_{HLS}$, there holds $$\mathcal{S}_{HLS}(g)\geq \frac{\delta}{2}\frac{n-2s}{n+2s}\min\{m(2\delta)\frac{n-2s}{n+2s},1\}.$$
Therefor we have established the stability of HLS inequality with explicit lower bounds for nonnegative functions.

\subsection{Stability of HLS inequality with explicit lower bounds}
In this subsection, we will prove the stability of HLS inequality with explicit lower bounds for general function $f\in L^{\frac{2n}{n+2s}}(\mathbb{R}^n)\setminus M_{HLS}$, namely we shall give the proof of Theorem \ref{thm1}. Let us denote by $C_{HLS}$ the optimal constant for stability of HLS inequality and denote by $C^{pos}_{HLS}$ the optimal constant in (\ref{sta hls func}) when restricted to nonnegative functions. In the spirit of the work by Dolbeault et al. in \cite{DEFFL}, we establish the relationship between these two optimal constants.

\begin{lemma}\label{constant}
$$C_{HLS}\geq \frac{1}{2}\min\{C^{pos}_{HLS},\min\{2^{\frac{n+2s}{n}}-2,1\}\}.$$
\end{lemma}
\begin{proof}
Let $D(g)=\|g\|^2_{\frac{2n}{n+2s}}-S_{n,s}^{-1}\|(-\Delta)^{-s/2}f\|_2^2$ and $f_{\pm}$ denote the positive and negative parts of $f$. Then
\begin{align}\label{est of D(g)}\nonumber
& D(f)\geq \|g\|^2_{\frac{2n}{n+2s}}-S_{n,s}^{-1}\|(-\Delta)^{-s/2}(g_{+})\|_2^2-S_{n,s}^{-1}\|(-\Delta)^{-s/2}(g_{-})\|_2^2\\
& =D(g_+)+D(g_-)+\|g\|^2_{\frac{2n}{n+2s}}-\|g_+\|^2_{\frac{2n}{n+2s}}-\|g_-\|^2_{\frac{2n}{n+2s}}.
\end{align}
Without loss of generality, we may assume
$$\|g\|^2_{\frac{2n}{n+2s}}=1,~~\text{and}~~m=\|g_{-}\|^{\frac{2n}{n+2s}}_{\frac{2n}{n+2s}}\in[0,1/2].$$
Let $h(m)=1-m^{\frac{n+2s}{n}}-(1-m)^{\frac{n+2s}{n}}$ and $t(m)=1-(1-m)^{\frac{n+2s}{n}}-2(1-(1/2)^{\frac{n+2s}{n}})m$. It is easy to check $t(0)=t(1/2)=0$ and
$t^{\prime\prime}(m)\leq 0$ on $[0,1/2]$, then $t(m)\geq 0$ on $[0,1/2]$, which means $1-(1-m)^{\frac{n+2s}{n}}\geq 2(1-(1/2)^{\frac{n+2s}{n}})m$. Thus
$$h(m)\geq 2(1-(1/2)^{\frac{n+2s}{n}})m-m^{\frac{n+2s}{n}},~~~m\in[0,1/2].$$
Since the function $2(1-(1/2)^{\frac{n+2s}{n}})m^{-\frac{2s}{n}}-1$ is decreasing on $[0,1/2]$, then
\begin{align}\label{est of h(m)}
h(m)\geq (2^{\frac{n+2s}{n}}-2)m^{\frac{n+2s}{n}}.
\end{align}
By (\ref{est of h(m)}) and sharp HLS inequality we have
$$\|g\|^2_{\frac{2n}{n+2s}}-\|g_+\|^2_{\frac{2n}{n+2s}}-\|g_-\|^2_{\frac{2n}{n+2s}}\geq (2^{\frac{n+2s}{n}}-2)\|g_-\|^2_{\frac{2n}{n+2s}}\geq
(2^{\frac{n+2s}{n}}-2)S_{n,s}^{-1}\|(-\Delta)^{s/2}g_{-}\|_2^2.$$
Along with (\ref{est of D(g)}),
$$D(g)\geq D(g_+)+D(g_{-})+(2^{\frac{n+2s}{n}}-2)S_{n,s}^{-1}\|(-\Delta)^{-s/2}g_{-}\|_2^2.$$
Let $h_{+}\in M_{HLS}$ such that $\|g_{+}-h_{+}\|_{\frac{2n}{n+2s}}=\inf\limits_{h\in M_{HLS}}\|g_{+}-h\|_{\frac{2n}{n+2s}}$. Since $$D(g_-)+S_{n,s}^{-1}\|(-\Delta)^{-s/2}(g_{-})\|_2^2=\|g_-\|^2_{\frac{2n}{n+2s}},$$ thus
\begin{align*}
& D(g)\geq D(g_+)+\min\{2^{\frac{n+2s}{n}}-2,1\}\|g_{-}\|^2_{\frac{2n}{n+2s}}\\
& \geq C^{pos}_{HLS}\|g_+-h_+\|^2_{\frac{2n}{n+2s}}+\min\{2^{\frac{n+2s}{n}}-2,1\}\|g_{-}\|^2_{\frac{2n}{n+2s}}\\
& \geq \min\{C^{pos}_{HLS},\min\{2^{\frac{n+2s}{n}}-2,1\}\}\frac{\left(\|g_+-h_+\|_{\frac{2n}{n+2s}}+\|g_{-}\|_{\frac{2n}{n+2s}}\right)^2}{2}\\
& \geq \frac{1}{2}\min\{C^{pos}_{HLS},\min\{2^{\frac{n+2s}{n}}-2,1\}\}\|g-h_+\|^2_{\frac{2n}{n+2s}}\\
& \geq \frac{1}{2}\min\{C^{pos}_{HLS},\min\{2^{\frac{n+2s}{n}}-2,1\}\}\inf_{h\in M_{HLS}}\|g-h\|^2_{\frac{2n}{n+2s}}
\end{align*}
which completes the proof of Lemma \ref{constant}.
\end{proof}

Now we are in position to prove Theorem \ref{thm1}. Denote by $$K_{n,s}=\sup\limits_{0<\delta<1}\frac{\delta}{2}\frac{n-2s}{n+2s}\min\{m(2\delta)\frac{n-2s}{n+2s},1\}.$$  We have already established $\mathcal{S}_{HLS}(g)\geq  K_{n,s}$
for all $0\leq g\in L^{\frac{2n}{n+2s}}(\mathbb{R}^n)\setminus M_{HLS}$  in subsection 3.2. This together with Lemma {\ref{constant}} yields that
for all $g\in L^{\frac{2n}{n+2s}}(\mathbb{R}^n)\setminus M_{HLS}$, there holds
$$\mathcal{S}_{HLS}(g)\geq \frac{1}{2}\min\{K_{n,s},\min\{2^{\frac{n+2s}{n}}-2,1\}\}.$$ Then we finish the proof of Theorem \ref{thm1}.

\section{HLS stability implies Sobolev stability}
In this section, we will prove the stability of fractional Sobolev inequality with explicit lower bounds for all $0<s<n/2$ from the stability of Hardy-Littlewood-Sobolev inequality with explicit lower bounds.
\medskip

Let $f\in \dot{H}^s(\mathbb{R}^n)$. Define $$\mathcal{F}(f)=S_{n,s}\|(-\Delta)^{s/2}f\|_2^2,~~~~\mathcal{E}(f)=\|f\|^2_{2^{\ast}_s}.$$
Then the Legendre transform $\mathcal{F}^{\ast}$ of a convex functional $\mathcal{F}: \dot{H}^s(\mathbb{R}^n)\rightarrow [0, +\infty)$ defined on $\dot{H}^{-s}(\mathbb{R}^n)$ is given by $$\mathcal{F}^{\ast}(g)=\sup_{f\in H^s(\mathbb{R}^n)}\{2\int_{\mathbb{R}^n}f(x)g(x)dx-\mathcal{F}(f)\}.$$
A simple calculation gives $\mathcal{F}^{\ast}(g)=S_{n,s}^{-1}\|(-\Delta)^{-s/2}g\|_2^2$. Similarly, the Legendre transform $\mathcal{E}^{\ast}$ of a convex functional $\mathcal{E}: L^{2^{*}_s}(\mathbb{R}^n)\rightarrow [0, +\infty)$ defined on $L^{\frac{2n}{n+2s}}(\mathbb{R}^n)$ is given by $$\mathcal{E}^{\ast}(g)=\sup_{f\in L^{2^{*}_s}(\mathbb{R}^n)}\{2\int_{\mathbb{R}^n}f(x)g(x)dx-\mathcal{E}(f)\}.$$
Obviously, $\mathcal{E}^{\ast}(g)=\|g\|^2_{\frac{2n}{n+2s}}$. Choose $g=\|f\|_{2^\ast_s}^{2-2^\ast_s}|f|^{2^\ast_s-1}sgn(f)$ and $f_1=S_{n,s}^{-1}(-\Delta)^{-s}g$, we can check that
\begin{align}\label{equation1}
\mathcal{E}(f)+\mathcal{E}^\ast(g)=2\int_{\mathbb{R}^n}fgdx
\end{align}
and
\begin{equation}\begin{split}\label{adequation2}
&\mathcal{F}(f)\\
&\ \ =S_{n,s}\|(-\Delta)^{s/2}f_1\|_2^2+S_{n,s}\|(-\Delta)^{s/2}(f-f_1)\|_2^2+2\int_{\mathbb{R}^n}(f-f_1)(S_{n,s}(-\Delta)^{s}f_1)dx\\
&\ \ =S_{n,s}^{-1}\|(-\Delta)^{-s/2}g\|_2^2+2\int_{\mathbb{R}^n}(f-S_{n,s}^{-1}(-\Delta)^{-s}g)gdx+S_{n,s}\|(-\Delta)^{s/2}(f-f_1)\|_2^2\\
&\ \ =2\int_{\mathbb{R}^n}fgdx-\mathcal{F}^{\ast}(g)+S_{n,s}\|(-\Delta)^{s/2}(f-f_1)\|_2^2.
\end{split}\end{equation}
Combining (\ref{equation1}) with (\ref{adequation2}), we have
\begin{align}\label{equation of deficit}
\mathcal{F}(f)-\mathcal{E}(f)= \mathcal{E}^\ast(g)-\mathcal{F}^\ast(g)+S_{n,s}\|(-\Delta)^{s/2}(f-f_1)\|_2^2.
\end{align}
Since we have already proved the stability of HLS inequality in Theorem \ref{thm1}
\begin{align}\label{est of hls def}
\mathcal{E}^\ast(g)-\mathcal{F}^\ast(g)\geq \frac{1}{2}\min\{K_{n,s},\min\{2^{\frac{n+2s}{n}}-2,1\}\}\inf\limits_{h\in M_{HLS}}\|g-h\|^2_{\frac{2n}{n+2s}},
\end{align}
then for any $\epsilon>0$, there exists
a $g_0\in M_{HLS}$ such that
$$\mathcal{E}^\ast(g)-\mathcal{F}^\ast(g)\geq \frac{1}{2}\min\{K_{n,s},\min\{2^{\frac{n+2s}{n}}-2,1\}\}\|g-g_0\|^2_{\frac{2n}{n+2s}}-\varepsilon+S_{n,s}\|(-\Delta)^{s/2}(f-f_1)\|_2^2.$$
Denote by $k_{n,s}=\frac{1}{2}\min\{K_{n,s},\min\{2^{\frac{n+2s}{n}}-2,1\}\}$, by (\ref{equation of deficit}), (\ref{est of hls def}), sharp HLS inequality, $(-\Delta)^{-s}g_0\in M_s$, we derive
\begin{align*}
& \mathcal{F}(f)-\mathcal{E}(f)\geq k_{n,s}S_{n,s}^{-1}\|(-\Delta)^{-s/2}(g-g_0)\|_2^2-\varepsilon+S_{n,s}\|(-\Delta)^{s/2}(f-f_1)\|_2^2\\
& =k_{n,s}\|S_{n,s}^{-1/2}(-\Delta)^{-s/2}(g-g_0)\|_2^2-\varepsilon+\|S_{n,s}^{1/2}(-\Delta)^{s/2}f-S_{n,s}^{-1/2}(-\Delta)^{-s/2}g\|_2^2\\
& \geq \frac{k_{n,s}}{2}\|S_{n,s}^{1/2}(-\Delta)^{s/2}f-S_{n,s}^{-1/2}(-\Delta)^{-s/2}g_0\|_2^2-\varepsilon\\
& \geq \frac{S_{n,s}k_{n,s}}{2}\inf_{h\in M_s}\|(-\Delta)^{s/2}(f-h)\|_2^2-\varepsilon,
\end{align*}
which completes the proof of Theorem \ref{thm2}.

\section{The proof of Theorem \ref{thm3}}
In this section, we will apply the stability of sharp HLS inequality obtained in Theorem \ref{thm1} to derive the lower-bound of the stability of sharp restrictive Sobolev inequality on the flat sub-manifold.
\medskip

Recall the fractional Sobolev inequality on the flat sub-manifold:
\begin{equation}\label{adsubmanifold}
C_{n,s}\int_{\mathbb{R}^n}|(-\Delta)^{\frac{s}{2}}u|^2dx\geq \big(\int_{\mathbb{R}^{n-1}}|u(x',0)|^{\frac{2(n-1)}{n-2s}}dx'\big)^{\frac{n-2s}{n-1}}.
\end{equation}
Let $(-\Delta)^{\frac{s}{2}}u=f$, the above inequality can be written as
\begin{equation}\label{adtrace-Sobolev}
\int_{\mathbb{R}^n}|f|^2dx\geq C_{n,s}^{-1}\big(\int_{\mathbb{R}^{n-1}}|T(f)(x',0)|^{\frac{2(n-1)}{n-2s}}dx'\big)^{\frac{n-2s}{n-1}},
\end{equation}
where $T(f)(x',0)=c_{n,s}\int_{\mathbb{R}^{n}}\frac{f(y',y_n)}{\big(|x'-y'|^2+|y_n|^2\big)^{\frac{n-s}{2}}}dy'dy_n$.
Obviously, the dual operator $T^{*}$ of $T$ can be written as
$$T^{*}(g)(x', x_n)=c_{n,s}\int_{\mathbb{R}^{n-1}}\frac{g(y')}{\big(|x'-y'|^2+|x_n|^2\big)^{\frac{n-s}{2}}}dy'.$$
Then inequality \eqref{adtrace-Sobolev} is equivalent to the following inequality
$$\big(\int_{\mathbb{R}^{n-1}}|g|^{\frac{2(n-1)}{n-2+2s}}dx'\big)^{\frac{n-2+2s}{n-1}}\geq C_{n,s}^{-1}\int_{\mathbb{R}^{n}}|T^{*}(g)(x',x_n)|^{2}dx'dx_n,$$
which reduces to
$$\big(\int_{\mathbb{R}^{n-1}}|g|^{\frac{2(n-1)}{n-2+2s}}dx'\big)^{\frac{n-2+2s}{n-1}}\geq S_{n-1,s-\frac{1}{2}}^{-1}\int_{\mathbb{R}^{n-1}}|(-\Delta)^{-\frac{s}{2}}g|^2dx'.$$
Hence, we see that the fractional Sobolev inequality on the flat sub-manifold is in fact can be seen as the dual form of Hardy-Littlewood-Sobolev inequality in $\mathbb{R}^{n-1}$. Now, we are in the position to apply dual stability and stability of the HLS inequality to calculate the lower bound of the stability of sharp restrictive Sobolev inequality on the flat sub-manifold \eqref{adsubmanifold}.
\vskip0.1cm

For $f\in \dot{H}^s(\mathbb{R}^n)$, define $$\mathcal{F}(f)=C_{n,s}\|(-\Delta)^{s/2}f\|_2^2,~~~~\mathcal{E}(f)=\big(\int_{\mathbb{R}^{n-1}}|f(x',0)|^{\frac{2(n-1)}{n-2s}}\big)^{\frac{n-2s}{n-1}}.$$
Then the Legendre transform $\mathcal{F}^{\ast}$ of a convex functional $\mathcal{F}$ is defined on the space $\dot{H}^{-s+\frac{1}{4}}(\mathbb{R}^{n-1})$ and is given by
\begin{equation*}\begin{split}
\mathcal{F}^{\ast}(g)&=\sup_{f\in H^s(\mathbb{R}^n)}\{2\int_{\mathbb{R}^{n-1}}f(x',0)g(x')dx'-\mathcal{F}(f)\}\\
&=\sup_{f\in H^s(\mathbb{R}^n)}\{2\int_{\mathbb{R}^n}(-\Delta)^{\frac{s}{2}}(f)T^{*}(g)dx-\mathcal{F}(f)\}.
\end{split}\end{equation*}
 Similarly, the Legendre transform $\mathcal{E}^{\ast}$ of a convex functional $\mathcal{E}$ is defined on $L^{\frac{2(n-1)}{n-2+2s}}(\mathbb{R}^{n-1})$ and is given by $$\mathcal{E}^{\ast}(g)=\sup_{f(x', 0)\in L^{\frac{2(n-1)}{n-2s}}(\mathbb{R}^{n-1})}\{2\int_{\mathbb{R}^{n-1}}f(x',0)g(x')dx'-\mathcal{E}(f)\}.$$
Simple calculation leads to $$\mathcal{F}^{\ast}(g)=C_{n,s}^{-1}\int_{\mathbb{R}^n}|T^{*}(g)|^2dx,\ \ \mathcal{E}^{\ast}(g)=\big(\int_{\mathbb{R}^{n-1}}|g(x')|^{\frac{2(n-1)}{n-2+2s}}dx'\big)^{\frac{n-2+2s}{n-1}}.$$ Choose $$g(x')=\|f(x',0)\|_{L^{\frac{2(n-1)}{n-2s}}(\mathbb{R}^{n-1})}^{2-\frac{2(n-1)}{n-2s}}|f(x',0)|^{\frac{2(n-1)}{n-2s}-1}sgn(f(x',0))$$ and $f_1=C_{n,s}^{-1}(-\Delta)^{-\frac{s}{2}}(T^{*}(g))$, we can calculate that
\begin{align*}\label{equation1}
\mathcal{E}(f)+\mathcal{E}^\ast(g)=2\int_{\mathbb{R}^{n-1}}f(x',0)g(x')dx'
\end{align*}
and
\begin{align*}
\mathcal{F}(f)&=C_{n,s}\int_{\mathbb{R}^n}|(-\Delta)^{\frac{s}{2}}f|^2dx\\
&=C_{n,s}\int_{\mathbb{R}^n}|(-\Delta)^{\frac{s}{2}}f_1|^2dx+C_{n,s}\int_{\mathbb{R}^n}|(-\Delta)^{\frac{s}{2}}(f-f_1)|^2dx\\
&\ \ +2C_{n,s}\int_{\mathbb{R}^n}|(-\Delta)^{\frac{s}{2}}(f-f_1)(-\Delta)^{\frac{s}{2}}(f_1)|dx\\
&=C_{n,s}^{-1}\int_{\mathbb{R}^n}|T^{*}(g)|^2dx+2\int_{\mathbb{R}^n}(-\Delta)^{\frac{s}{2}}(f-f_1)T^{*}(g)dx+C_{n,s}\int_{\mathbb{R}^n}|(-\Delta)^{\frac{s}{2}}(f-f_1)|^2dx\\
&=C_{n,s}^{-1}\int_{\mathbb{R}^n}|T^{*}(g)|^2dx+2\int_{\mathbb{R}^{n-1}}(f-f_1)(x',0)g(x')dx'+C_{n,s}\int_{\mathbb{R}^n}|(-\Delta)^{\frac{s}{2}}(f-f_1)|^2dx\\
& =2\int_{\mathbb{R}^{n-1}}f(x',0)g(x')dx'-\mathcal{F}^{\ast}(g)+C_{n,s}\int_{\mathbb{R}^n}|(-\Delta)^{\frac{s}{2}}(f-f_1)|^2dx.
\end{align*}
Then it follows that
\begin{equation*}
\mathcal{F}(f)-\mathcal{E}(f)= \mathcal{E}^\ast(g)-\mathcal{F}^\ast(g)+C_{n,s}\int_{\mathbb{R}^n}|(-\Delta)^{\frac{s}{2}}(f-f_1)|^2dx.
\end{equation*}
Note that $$\mathcal{E}^\ast(g)-\mathcal{F}^\ast(g)=\big(\int_{\mathbb{R}^{n-1}}|g|^{\frac{2(n-1)}{n-2+2s}}dx'\big)^{\frac{n-2+2s}{n-1}}- S_{n-1,s-\frac{1}{2}}^{-1}\int_{\mathbb{R}^{n-1}}|(-\Delta)^{-\frac{s}{2}+\frac{1}{4}}g|^2dx'.$$
Since we have already proved the stability of HLS inequality in Theorem \ref{thm1}:
\begin{align*}\label{est of hls def}
\mathcal{E}^\ast(g)-\mathcal{F}^\ast(g)\geq \frac{1}{2}\min\{K_{n-1,s-\frac{1}{2}},\min\{2^{\frac{n-2+2s}{n-1}}-2,1\}\}\inf\limits_{h\in M_{HLS}}\|g-h\|^2_{\frac{2(n-1)}{n-2+2s}},
\end{align*}
then for any $\epsilon>0$, there exists
a $g_0\in M_{HLS}$ such that
$$\mathcal{E}^\ast(g)-\mathcal{F}^\ast(g)\geq \frac{1}{2}\min\{K_{n-1,s-\frac{1}{2}},\min\{2^{\frac{n-2+2s}{n-1}}-2,1\}\}\|g-g_0\|^2_{\frac{2(n-1)}{n-2+2s}}-\varepsilon+C_{n,s}\|(-\Delta)^{s/2}(f-f_1)\|_2^2.$$
Denote by $k_{n-1,s-\frac{1}{2}}=\frac{1}{2}\min\{K_{n-1,s-\frac{1}{2}},\min\{2^{\frac{n-2+2s}{n-1}}-2,1\}\}$, we can write
\begin{align*}
& \mathcal{F}(f)-\mathcal{E}(f)\geq k_{n-1,s-\frac{1}{2}}C_{n,s}^{-1}\int_{\mathbb{R}^n}|T^{*}(g-g_0)|^2dx-\varepsilon+C_{n,s}\int_{\mathbb{R}^n}|(-\Delta)^{s/2}(f-f_1)|^2dx\\
& =k_{n-1,s-\frac{1}{2}}\int_{\mathbb{R}^n}|C_{n,s}^{-1/2}T^{*}(g-g_0)|^2dx-\varepsilon+\int_{\mathbb{R}^n}|C_{n,s}^{1/2}(-\Delta)^{s/2}f-C_{n,s}^{-1/2}T^{*}(g)|^2dx\\
& \geq \frac{k_{n-1,s-\frac{1}{2}}}{2}\int_{\mathbb{R}^n}|C_{n,s}^{1/2}(-\Delta)^{s/2}f-C_{n,s}^{-1/2}T^{*}(g_0)|^2dx-\varepsilon\\
& \geq \frac{C_{n,s}k_{n-1,s-\frac{1}{2}}}{2}\inf_{h\in M_{\mathbb{R}^{n-1}}}\|(-\Delta)^{s/2}(f-h)\|_2^2-\varepsilon,
\end{align*}
where we have used the fact $(-\Delta)^{-\frac{s}{2}}(T^{*}(g_0))\in M_{\mathbb{R}^{n-1}}$. Then we accomplish the proof of Theorem \ref{thm3}.

\section{The proof of Theorem \ref{thm4}}
In this section, we will give the stability of restrictive Sobolev inequality on the sphere $\mathbb{S}^{n-1}$, namely we shall give the proof of Theorem \ref{thm4}.
\medskip

Recall the sharp restrictive Sobolev inequality on the sphere $\mathbb{S}^{n-1}$:
$$D_{n,s}\int_{\mathbb{R}^n}|(-\Delta)^{\frac{s}{2}}u|^2dx\geq \big(\int_{\mathbb{S}^{n-1}}|u(\xi)|^{\frac{2(n-1)}{n-2s}}d\sigma_{\xi}\big)^{\frac{n-2s}{n-1}}.$$

Let $(-\Delta)^{\frac{s}{2}}u=f$, the above inequality can be written as
\begin{equation}\label{trace-Sobolev}
\int_{\mathbb{R}^n}|f|^2dx\geq C_{n,s}\big(\int_{\mathbb{S}^{n-1}}|\widetilde{T}(f)(\xi)|^{\frac{2(n-1)}{n-2s}}d\sigma_{\xi}\big)^{\frac{n-2s}{n-1}},
\end{equation}
where $\widetilde{T}(f)(\xi)=c_{n,s}\int_{\mathbb{R}^{n}}\frac{f(y)}{|\xi-y|^{n-s}}dy$.
Using duality, the above inequality \eqref{trace-Sobolev} is equivalent to the following inequality
$$\big(\int_{\mathbb{S}^{n-1}}|g(\xi)|^{\frac{2(n-1)}{n-2+2s}}d\sigma_{\xi}\big)^{\frac{n-2+2s}{n-1}}\geq D_{n,s}\int_{\mathbb{R}^{n}}|\widetilde{T}^{*}(g)(x)|^{2}dx,$$
where $\widetilde{T}^{*}(g)(x)=c_{n,s}\int_{\mathbb{S}^{n-1}}\frac{g(\xi)}{|x-\xi|^{n-s}}d\sigma_{\xi}$. Now we start to establish the stability of the sharp restrictive Sobolev inequality on the sphere $\mathbb{S}^{n-1}$. Let $f\in \dot{H}^s(\mathbb{R}^n)$, define $$\mathcal{F}(f)=D_{n,s}\|(-\Delta)^{s/2}f\|_2^2,~~~~\mathcal{E}(f)=\big(\int_{\mathbb{S}^{n-1}}|f(\xi)|^{\frac{2(n-1)}{n-2s}}d\sigma_{\xi}\big)^{\frac{n-2s}{n-1}}.$$
Then the Legendre transform $\mathcal{F}^{\ast}$ of a convex functional $\mathcal{F}$ is defined on $H^{-s+\frac{1}{2}}(\mathbb{S}^{n-1})$ and is given by
\begin{equation}\begin{split}
\mathcal{F}^{\ast}(g)&=\sup_{f\in H^s(\mathbb{R}^n)}\{2\int_{\mathbb{S}^{n-1}}f(\xi)g(\xi)d\sigma_{\xi}-\mathcal{F}(f)\}\\
&=\sup_{f\in H^s(\mathbb{R}^n)}\{2\int_{\mathbb{R}^n}(-\Delta)^{\frac{s}{2}}(f)\widetilde{T}^{*}(g)dx-\mathcal{F}(f)\}
\end{split}\end{equation}
Similarly, the Legendre transform $\mathcal{E}^{\ast}$ of a convex functional is defined on $L^{\frac{2(n-1)}{n-2+2s}}(\mathbb{S}^{n-1})$ and is given by  $$\mathcal{E}^{\ast}(g)=\sup_{f\in L^{\frac{2(n-1)}{n-2s}}(\mathbb{S}^{n-1})}\{2\int_{\mathbb{S}^{n-1}}f(\xi)g(\xi)d\sigma_{\xi}-\mathcal{E}(f)\}.$$
It is not difficult to check that $\mathcal{F}^{\ast}(g)=D_{n,s}^{-1}\int_{\mathbb{R}^n}|\widetilde{T}^{*}(g)|^2dx$ and
$\mathcal{E}^{\ast}(g)=\big(\int_{\mathbb{S}^{n-1}}|g(\xi)|^{\frac{2(n-1)}{n-2+2s}}d\sigma_{\xi}\big)^{\frac{n-2+2s}{n-1}}$. Pick $$g(\xi)=\|f(\xi)\|_{L^{\frac{2(n-1)}{n-2s}}(\mathbb{S}^{n-1})}^{2-\frac{2(n-1)}{n-2s}}|f(\xi)|^{\frac{2(n-1)}{n-2s}-1}sgn(f(\xi))$$
and $f_1=D_{n,s}^{-1}(-\Delta)^{-\frac{s}{2}}(\widetilde{T}^{*}(g))$, we can write
\begin{equation*}\begin{split}\label{equation2}\nonumber
&\mathcal{F}(f)=D_{n,s}\int_{\mathbb{R}^{n}}|(-\Delta)^{s/2}f|^2dx\\
&=D_{n,s}\int_{\mathbb{R}^n}|(-\Delta)^{\frac{s}{2}}f_1|^2dx+D_{n,s}\int_{\mathbb{R}^n}|(-\Delta)^{\frac{s}{2}}(f-f_1)|^2dx\\
&\ \ +2D_{n,s}\int_{\mathbb{R}^n}|(-\Delta)^{\frac{s}{2}}(f-f_1)(-\Delta)^{\frac{s}{2}}(f_1)|dx\\
&=D_{n,s}^{-1}\int_{\mathbb{R}^n}|\widetilde{T}^{*}(g)|^2dx+2D_{n,s}\int_{\mathbb{R}^n}(-\Delta)^{\frac{s}{2}}(f-f_1)\widetilde{T}^{*}(g)dx+D_{n,s}\int_{\mathbb{R}^n}|(-\Delta)^{\frac{s}{2}}(f-f_1)|^2dx\\
&=D_{n,s}^{-1}\int_{\mathbb{R}^n}|\widetilde{T}^{*}(g)|^2dx+2\int_{\mathbb{S}^{n-1}}(f-f_1)(\xi)g(\xi)d\sigma_{\xi}+D_{n,s}\int_{\mathbb{R}^n}|(-\Delta)^{\frac{s}{2}}(f-f_1)|^2dx\\
& =2\int_{\mathbb{S}^{n-1}}f(\xi)g(\xi)d\sigma_{\xi}-\mathcal{F}^{\ast}(g)+D_{n,s}\int_{\mathbb{R}^n}|(-\Delta)^{\frac{s}{2}}(f-f_1)|^2dx.
\end{split}\end{equation*}
It is easy to check that
\begin{align}\label{addequation1}
\mathcal{E}(f)+\mathcal{E}^\ast(g)=2\int_{\mathbb{S}^{n-1}}f(\xi)g(\xi)d\sigma_{\xi}.
\end{align}
Hence it follows that
\begin{align*}\label{equation of deficit}
\mathcal{F}(f)-\mathcal{E}(f)= \mathcal{E}^\ast(g)-\mathcal{F}^\ast(g)+D_{n,s}\int_{\mathbb{R}^n}|(-\Delta)^{\frac{s}{2}}(f-f_1)|^2dx.
\end{align*}

Note that $$\mathcal{E}^\ast(g)-\mathcal{F}^\ast(g)=\big(\int_{\mathbb{S}^{n-1}}|g(\xi)|^{\frac{2(n-1)}{n-2+2s}}d\sigma_\xi\big)^{\frac{n-2+2s}{n-1}}-B_{n-1,s-\frac{1}{2}}\int_{\mathbb{S}^{n-1}}\int_{\mathbb{S}^{n-1}}g(\xi)|\xi-\eta|^{-(n-2s)}g(\eta)d\sigma_\xi d\sigma_\eta.$$
Since we have already proved the stability of HLS inequality on the sphere in Corollary \ref{coro1}, hence we derive that
\begin{equation*}
\mathcal{E}^\ast(g)-\mathcal{F}^\ast(g)\geq \frac{1}{2}\min\{K_{n-1,s-\frac{1}{2}},\min\{2^{\frac{n-2+2s}{n-1}}-2,1\}\} \inf\limits_{h\in \widetilde{M}_{HLS}} \big(\int_{\mathbb{S}^{n-1}}|g-h|^{\frac{2(n-1)}{n-2+2s}}d\sigma_{\xi}\big)^{\frac{n-2+2s}{n-1}}.
\end{equation*}
Then it follows that for any $\epsilon>0$, there exists
a $g_0\in M_{HLS}$ such that
$$\mathcal{E}^\ast(g)-\mathcal{F}^\ast(g)\geq \frac{1}{2}\min\{K_{n-1,s-\frac{1}{2}},\min\{2^{\frac{n-2+2s}{n-1}}-2,1\}\}\|g-g_0\|^2_{\frac{2(n-1)}{n-2+2s}}-\varepsilon+D_{n,s}\|(-\Delta)^{s/2}(f-f_1)\|_2^2.$$
Denote by $k_{n-1,s-\frac{1}{2}}=\frac{1}{2}\min\{K_{n-1,s-\frac{1}{2}},\min\{2^{\frac{n-2+2s}{n-1}}-2,1\}\}$, we obtain
\begin{align*}
& \mathcal{F}(f)-\mathcal{E}(f)\geq k_{n-1,s-\frac{1}{2}}D_{n,s}^{-1}\int_{\mathbb{R}^n}|\widetilde{T}^{*}(g-g_0)|^2dx-\varepsilon+D_{n,s}\int_{\mathbb{R}^n}|(-\Delta)^{s/2}(f-f_1)|^2dx\\
& =k_{n-1,s-\frac{1}{2}}\int_{\mathbb{R}^n}|D_{n,s}^{-1/2}\widetilde{T}^{*}(g-g_0)|^2dx-\varepsilon+\int_{\mathbb{R}^n}|D_{n,s}^{1/2}(-\Delta)^{s/2}f-D_{n,s}^{-1/2}\widetilde{T}^{*}(g)|^2dx\\
& \geq \frac{k_{n-1,s-\frac{1}{2}}}{2}\int_{\mathbb{R}^n}|D_{n,s}^{1/2}(-\Delta)^{s/2}f-D_{n,s}^{-1/2}\widetilde{T}^{*}(g_0)|^2dx-\varepsilon\\
& \geq \frac{D_{n,s}k_{n-1,s-\frac{1}{2}}}{2}\inf_{h\in M_{\mathbb{S}^{n-1}}}\|(-\Delta)^{s/2}(f-h)\|_2^2-\varepsilon,
\end{align*}
which completes the proof of Theorem \ref{thm4}.

\bibliographystyle{amsalpha}

\end{document}